\documentclass{article}
\usepackage{amsfonts}
\usepackage{latexsym,amsfonts,amsmath,amsthm,makeidx}
\usepackage{mathrsfs}
\usepackage{makeidx}
\usepackage{color}
\makeindex
\newtheorem{definition}{Definition}[section]
\newtheorem{theorem}{Theorem}[section]
\newtheorem{lemma}{Lemma}[section]
\newtheorem{corollary}{Corollary}[section]
\newtheorem{proposition}{Proposition}[section]
\newtheorem{remark}{Remark}[section]

\newcommand{\bpp}{\begin{proposition}}
\newcommand{\epp}{\end{proposition}}

\newcommand{\bt}{\begin{theorem}}
\newcommand{\et}{\end{theorem}}
\newcommand{\bl}{\begin{lemma}}
\newcommand{\el}{\end{lemma}}
\newcommand{\bd}{\begin{definition}}
\newcommand{\ed}{\end{definition}}
\newcommand{\bc}{\begin{corollary}}
\newcommand{\ec}{\end{corollary}}
\newcommand{\bp}{\begin{proof}}
\newcommand{\ep}{\end{proof}}
\newcommand{\bx}{\begin{example}}
\newcommand{\ex}{\end{example}}
\newcommand{\bi}{\begin{exercise}}
\newcommand{\ei}{\end{exercise}}
\newcommand{\bo}{\begin{prop}}
\newcommand{\eo}{\end{prop}}
\newcommand{\br}{\begin{remark}}
\newcommand{\er}{\end{remark}}
\newcommand{\be}{\begin{equation}}
\newcommand{\ee}{\end{equation}}
\newcommand{\ba}{\begin{align}}
\newcommand{\ea}{\end{align}}
\newcommand{\bn}{\begin{enumerate}}
\newcommand{\en}{\end{enumerate}}
\newcommand{\bg}{\begin{align*}}
\newcommand{\bcs}{\begin{cases}}
\newcommand{\ecs}{\end{cases}}

\newcommand{\bean}{\begin{eqnarray*}}
\newcommand{\eean}{\end{eqnarray*}}

\numberwithin{equation}{section}

\begin{document}

\title{\bf Soliton solutions for a class of quasilinear Schr\"{o}dinger equations with  a   parameter\thanks{Supported by NSFC (No.11201154,11371146), SRFDF(No.20120172120040) e-mails: scyjwang@scut.edu.cn.(Wang); maytshen@scut.edu.cn (Shen);
 }
\date{}
\author{Claudianor  O. Alves\thanks{Supported by CNPq/Brazil 303080/2009-4, e-mail:coalves@dme.ufcg.edu.br} \\
{\small \it Unidade Acad\^{e}mica de Matem\'{a}tica,}\\{ \small \it Universidade Federal de Campina Grande, }{\small \it58429-900, Campina Grande-PB-Brazil}\\
{  Youjun Wang \quad Yaotian Shen}\\ {\small\it Department of  Mathematics, South China University of
 Technology,}
\\{\small\it Guangzhou 510640, PR China}\\
}}

\maketitle
\begin{center}{\it   }\end{center}

\vskip0.36in

\begin{center}
\begin{minipage}{120mm}
{\bf Abstract }
 Using variational methods combined with perturbation arguments,  we study   the existence of nontrivial classical solution  for the quasilinear Schr\"{o}dinger equation
 \begin{equation*}\label{1.1}
 -\Delta u+ V(x)u+ \frac{\kappa}{2}[\Delta |u|^2]u=l(u),\ x\in\mathbb{R}^N,
 \end{equation*}
where  $V:\mathbb{R}^N\rightarrow \mathbb{R}$ and $l:\mathbb{R} \to \mathbb{R}$ are continuous function,  $\kappa$ is a parameter and  $N\geq 3$. This model has been proposed in  plasma physics and nonlinear optics. As a main novelty with respect to some
previous results, we are able to deal with the case $\kappa>0$.
\end{minipage}
\end{center}
 \vspace{0.5cm}

\vskip0.6in
\section{Introduction}	

In this paper we   consider the following
quasilinear Schr\"{o}dinger equations
\begin{equation}\label{00.1}
-\Delta u+ V(x)u+ \frac{\kappa}{2}[\Delta |u|^2]u=l(u),\ x\in\mathbb{R}^N,
\end{equation}
where  $V:\mathbb{R}^N\rightarrow \mathbb{R}$ and $l:\mathbb{R} \to \mathbb{R}$ are continuous function and $\kappa>0$ is a parameter.
Solutions of (\ref{00.1})  are
related to standing wave solutions for the following Schr\"{o}dinger equations:
\begin{equation}\label{1.2}
iz_t=-\Delta z+W(x) z-\rho(|z|^2)z+ \frac{\kappa}{2}[\Delta
|z|^2]z,\ \ x\in \mathbb{R}^N,
\end{equation}
where $W:\mathbb{R}^N\rightarrow \mathbb{R}$ is a
given potential and $\rho:\mathbb{R} \rightarrow \mathbb{R}$ is a real function. Quasilinear Schr\"{o}dinger equations  like (\ref{1.2}) play an important  role in various domains in physics. For $\rho(s)=as^{p}$, (\ref{1.2}) appears in various problems in plasma physics and nonlinear optics, e.g. oscillating soliton instabilities during microwave and laser heating of plasma \cite{Go-84,Po-78}. Moreover, (\ref{1.2}) is also the basic equation describing oscillations in a superfluid film when $\rho(s)=-\alpha-\frac{\beta}{(a+s)^3}$ \cite{Ku-81}.  We   refer
the readers to \cite{Br-La, Bo-Ha, Hasse, La-Po, Lit} for more details on the background.

In this paper, we restrict ourselves to two model cases $\rho(s)=s^{\frac{q-2}{2}}$ or  \linebreak $\rho(s)=-1+\frac{1}{(1+s)^3}$ and  we are interested in the existence of standing wave  solutions, i.e. solutions of the form
$z(t,x)=\exp(-iEt)u(x)$,
where $E\in \mathbb{R}$ and $u$ is a real function. Putting $z(t,x)=\exp(-iEt)u(x)$ into (\ref{1.2}),  we
are led to  the following  equations
\begin{equation}\label{1.1}
-\Delta u+ V(x)u+ \frac{\kappa}{2}[\Delta |u|^2]u=|u|^{q-2}u,\ x\in\mathbb{R}^N
\end{equation}
or \begin{equation}\label{1.1++}
-\Delta u+ V(x)u+ \frac{\kappa}{2}[\Delta |u|^2]u=\Big[1-\frac{1}{(1+|u|^2)^3}\Big]u,\ x\in\mathbb{R}^N
\end{equation}
with $V(x)=W(x)-E$.

For equation (\ref{1.1}),  semilinear case corresponding to $\kappa=0$
has been studied extensively in recent years, see e.g. \cite{Ber-Lions,Lions}.   When $\kappa<0$, this equation has been introduced in \cite{Br-Er-01, Br-Er-03, Ha-Za-03} to study a model of self-trapped electrons in
quadratic or hexagonal lattices  and
has attracted much attention.  For the subcritical case, i.e., $4<q<22^*$ in (\ref{1.1}), the first existence results  are, up to our knowledge,
due to Poppenberg, Schmitt and  Wang in \cite{Pop-Sch-Wang}. In \cite{Pop-Sch-Wang}, the main existence results are obtained, through
a constrained minimization argument.
Subsequently a general existence result for (\ref{1.1}) was derived in Liu, Wang and Wang \cite{Liu-3}. The idea in \cite{Liu-3} is to make  a change of variable and
reduce the quasilinear problem (\ref{1.1}) to semilinear one and an Orlicz
space framework was used to prove the existence of a positive
solutions
via Mountain pass theorem.
The same method of changing of variable
was also used by Colin and Jeanjean in  \cite{Colin-Jean}, but the usual Sobolev space $H^1(\mathbb{R}^N)$
framework was used as the working space. Precisely,  since the energy functional   associated
to  (\ref{1.1}) is not well defined in $H^1(\mathbb{R}^N)$,  they first make the changing of  unknown variables
$v=f^{-1}(u)$, where $f$ is defined by ODE:
\begin{equation}\label{0.01} f'(t)=\frac{1}{\sqrt{1-\kappa f^2(t)}},\quad  t\in [0,+\infty),\end{equation}
and
$f(t)=-f(-t),$ $ t\in (-\infty, 0].$ Then,  after the changing of variables, to find the solutions of (\ref{1.1}), it suffices to study the existence of solutions for the following semilinear equation
$$-\Delta v=\frac{1}{\sqrt{1- \kappa f^2(v)}}(-V(x)f(v)+|f(v)|^{q-2}f(v)),\quad x\in \mathbb{R}^N.$$
By using the classical
results given by Berestycki and  Lions \cite{Ber-Lions}, they proved the
existence of a spherically symmetric solution.  In \cite{Liu-Wang-Wang-2004},  the authors used a minimization on a Nehari-type constraint to get existence
results. Their argument does not depend on any change of variables, so it can be applied to treat
more general problems. By minimization under a convenient constraint,  Ruiz and Siciliano  in \cite{Ruiz-Si-2010} discussed the  existence of ground states for (\ref{1.1}) with $q\in (2,\frac{4N}{N-2})$. For  the critical case,
Silva and Vieira in \cite{Silva} established the existence  of solutions for asymptotically periodic quasilinear
Schr\"{o}dinger equations (\ref{1.1}) with the nonlinearity $|u|^{q-2}u$ replaced by $K(x)u^{2(2^*)-1}+ g(x,u)$. The existence of multiple solutions  were established in  \cite{Wang-Yang}.  We refer to \cite{do-09, Liu-2, Wang-Yang, Wang-Zou} for more results.

Recently, in  \cite{Shen-2013,  Wang-Yang-13}, the authors introduced the changing of  known variables
$s=G^{-1}(t)$ for $t\in [0,+\infty)$ and $G^{-1}(t)=-G^{-1}(-t)$ for $t\in (-\infty,0)$, where\begin{equation}\label{0.02}G(s)=\int_0^s \sqrt{1-\kappa t^2}dt.\end{equation}
Since $\kappa<0$, integral (\ref{0.02}) makes sense and  the inverse function $G^{-1}(t)$  exists. Then, using variational methods,  they established the existence of nontrivial solutions for (\ref{1.1}) with subcritical or critical growth.

The main purpose of the present paper is studying  the existence of nontrivial solutions for models (\ref{1.1}) and (\ref{1.1++}) with $\kappa>0$. Unfortunately, we  note that at this moment, neither the  changing of   variables (\ref{0.01}) nor (\ref{0.02}) are  suitable for dealing with this kind of problem because $1-\kappa t^2$ may be negative. As far as we know,  in the mathematical literature, few results are known on (\ref{1.1})  and  (\ref{1.1++}) with $\kappa>0$.
Hereafter, we assume that potential $V:\mathbb{R}^N\rightarrow \mathbb{R}$  is continuous and satisfies: 
\begin{itemize}
\item [$(V_0)$] $V(x)\geq V_0>0, \ \text{ for all}\ x\in\mathbb{R}^N.$
\item [$(V_1)$] $\lim\limits_{|x|\rightarrow \infty}V(x)=V_\infty
$
and $V(x)\leq V_\infty,$  $\text{ for all}\ x\in\mathbb{R}^N.$
\end{itemize}

We have the following result:

\begin{theorem}
\label{th1.1}
Assume that  $2<q<2^*$, $(V_0)$ and $(V_1)$.  Then,  there  exists  some $\kappa_0>0$ such that  for all $\kappa\in [0,\kappa_0)$,  (\ref{1.1}) has a solution. Moreover,   $\max\limits_{x\in \mathbb{R}^N}|u(x)|\leq \sqrt{\dfrac{1}{\kappa}}$.
\end{theorem}

For equation (\ref{1.1++}), we may state:
\begin{theorem}
\label{th1.2}
Assume that   $(V_0)$ with  $V_0\geq1$ and $(V_1)$.  Then,  there  exists  some $\kappa_1\in (0,\frac{1}{3})$ such that  for all $\kappa\in (0,\kappa_1)$,  (\ref{1.1++}) has a solution. Moreover,   $\max\limits_{x\in \mathbb{R}^N}|u(x)|\leq 1$.
\end{theorem}

\begin{remark}
When $V(x)\equiv V_\infty$, using the classical
results given by Berestycki and  Lions \cite{Ber-Lions}, Theorems  \ref{th1.1} and \ref{th1.2} are still true. Furthermore, $u$ also has the following properties:
\begin{itemize}
\item[$(1)$] $u>0$ on $\mathbb{R}^N;$
\item[$(2)$] $u$ is spherically symmetric  and $u$ decreases
with respect to $|x|$;
\item[$(3)$] $u\in C^2(\mathbb{R}^N);$
\item[$(4)$]  $u$ together with its derivatives up to order 2 have exponential decay at infinity:$$|D^\alpha u|\leq Ce^{-\delta|x|},\quad x\in \mathbb{R}^N,$$
for some $C$, $\delta>0$ and $|\alpha|\leq2$.
\end{itemize}
Therefore, in this paper, we assume that $V(x)\leq V_\infty$ for all $x\in \mathbb{R}^N$ but $V(x)\not\equiv V_\infty$.

\end{remark}

\begin{remark}
When $\kappa=0$, (\ref{00.1}) has already been studied by many authors, see e.g. \cite{CO-SE-SO-11, Ber-Lions, Jean-Tan}. So, in  Theorem \ref{th1.1}, we only consider the case $\kappa>0$.
\end{remark}

\begin{remark}
When $\kappa=0$,  equation (\ref{1.1++}) turns into the following asymptotically semilinear problem  \begin{equation*}
-\Delta u+ V(x)u=\Big[1-\frac{1}{(1+|u|^2)^3}\Big]u,\ x\in\mathbb{R}^N.
\end{equation*}
That is, the nonlinearity $\rho(t)=\big[1-\frac{1}{(1+|t|^2)^3}\big]t$ satisfy $\lim\limits_{t\rightarrow \infty}\frac{\rho(t)}{t}=1.$ Although the existence of nontrivial solutions  for this type of equation may already be known, we have been unable to find a proper reference.
\end{remark}

\begin{remark}
We remark that in Theorem \ref{th1.1}, $\kappa_0$  is dependent on the value $2<q<2^*$.
\end{remark}

\begin{remark}
In \cite{Br-La}, L. Br\"{u}ll, H. Lange and K\"{o}ln  studied the one-dimensional quasilinear Schr\"{o}dinger equations
\begin{equation}\label{1.2-}
iz_t=-\partial_x^2 z-|z|^{2p}z+ \kappa\partial_x^2
(|z|^2)z,\ \ x\in \mathbb{R}
\end{equation}
and
\begin{equation}\label{1.3-}
iz_t=-\partial_x^2 z-\Big[\mu+\dfrac{A}{(a+|z|^2)^3}\Big]z+ \kappa\partial_x^2
(|z|^2)z,\ \ x\in \mathbb{R},
\end{equation}
where $z=z(x,t)$ is the unknown wave function,  $\kappa$ is a real constant, $p>0$, $\mu>0$ and $A<0$. Under some conditions on $p$, $\mu$ and $A$, they proved that if $0<\kappa<\kappa_2$ (or $0<\kappa<\kappa_3$)  with some $\kappa_2, \kappa_3>0$,  then (\ref{1.2-}) (or (\ref{1.3-})) has a standing wave solution $v(x)$ with $v(x)>0$, $v(-x)=v(x)$, $v'(x)<0$ for $x>0$ and $\lim\limits_{|x|\rightarrow \infty}v(x)=0.$ Moreover, this solution is unique up to translation. Here, we generalize their results  to higher dimension.
\end{remark}
\begin{remark}
For $\kappa>0$, H. Lange,  M. Poppenberg and  H. Teisniann  \cite{La-Po} studied the whole space Cauchy problem for quasilinear Schr\"{o}dinger equation (\ref{1.2}) with $W=0$ and $\rho=0$. When $N=1$ and $z(0,x)=\phi(x)$, they obtained $L^2$$-$solutions for (\ref{1.2}) with  $\kappa|\phi(x)|\leq\delta<1$. Moreover, for $2\kappa\|\phi\|_{W^{1,\infty}}<1$, they also proved the existence of $H^2$$-$solutions for arbitrary space
dimension. We refer to  \cite{La-Po} for more details.
\end{remark}


Note that (\ref{1.1}) is the Euler-Lagrange equation associated to the natural energy
functional
\begin{equation}\label{2.1}I(u)=\frac{1}{2}\int_{\mathbb{R}^N}(1-\kappa u^2)|\nabla u|^2dx+\frac{1}{2}\int_{\mathbb{R}^N}V(x)u^2dx-\frac{1}{q}\int_{\mathbb{R}^N}|u|^qdx
\end{equation}
From the variational point of view, the first difficulty that we have to deal with
is to find some proper Sobolev space  since (\ref{2.1}) is not well defined in $H^1(\mathbb{R}^N)$ for $N\geq3$ and $\kappa\not=0.$ However, even if this difficulty is set up, there is another one: to guarantee the positiveness of the  principal part, i.e. $1-\kappa u^2>0$.

In  order to prove our main results, we first establish a nontrivial solution  for  a modified quasilinear Schr\"{o}dinger equation. Precisely, we consider the existence of  nontrivial solutions for the following  quasilinear
Schr\"{o}dinger equation
\begin{equation}\label{1.0}
-div(g^2(u)\nabla u)+g(u)g'(u)|\nabla u|^2+V(x)u=l(u),\ \ x\in
\mathbb{R}^N
\end{equation}
with $g(t)=\sqrt{1-\kappa t^2}$ for $|t|<\sqrt{\dfrac{1}{3\kappa}}$ for $\kappa>0$ ,  where  $V:\mathbb{R}^N\rightarrow \mathbb{R}$ is a continuous function, $2<q<2^*$, $N\geq3.$  Clearly, when $g(t)=\sqrt{1-\kappa t^2}$ and $l(u)=|u|^{q-2}u$, (\ref{1.0}) turns into (\ref{1.1}).
Then, by using Morse $L^\infty$ estimate, we prove that there exists $\kappa_0>0$ such that for all $\kappa\in [0,\kappa_0)$ the solutions that we have found  verify the estimate $\max \limits_{\mathbb{R}^N}|u|<\sqrt{\dfrac{1}{3\kappa}}$. Thus, they are solutions of the original problem (\ref{1.1}). To prove  Theorem \ref{th1.2}, we need further to  modify the  nonlinearity.

We mention that  similar method have been adopted by Alves,  Soares and Souto  to study a supercritical Schr\"{o}dinger-Poisson equation \cite{CO-SE-SO-11}. In
\cite{CO-SE-SO-11}, they mainly modified the nonlinearity and provide an estimate involving the $L^\infty-$norm of a solution related to a subcritical problem. However, unlike \cite{CO-SE-SO-11}, here we need to modify the  principal part first.

The organization of this paper is as follows: In Section 2, we  reformulate  the
problem and study the existence of nontrivial solutions of a modified quasilinear Schr\"{o}dinger equation (\ref{1.0}). In Section 3, we provide an estimate involving the $L^\infty$-norm of a solution related to (\ref{1.0}) and we  prove Theorems 1.1.  Section 4  is devoted to prove Theorem \ref{th1.2}.

 In this paper,  $C$, $C_i$, $i=1,2,\cdots$ denote positive (possibly
different) constant. Moreover, $\|\cdot\|_p$ denotes the norm of $L^p(\mathbb{R}^N)$.

\section{The modified problem}

Hereafter, we shall work on  the space $H^{1}(\mathbb{R}^N)$ endowed with the norm
$$\|u\|=\Bigg[\int_{\mathbb{R}^N}(|\nabla u|^2+V(x)u^2)dx\Bigg]^{\frac{1}{2}}.$$
By $(V_0)$ and $(V_1)$, the above norm is equivalent the usual one on $H^{1}(\mathbb{R}^{N})$ .

For equation (\ref{1.0}), we let $l(t)=|t|^{q-2}t$ for $2<q<2^*$ and  we will consider $g:[0, +\infty) \to \mathbb{R}$ given by
\begin{equation*}g(t)=\begin{cases}
\sqrt{1-\kappa t^2},&\quad \text{ if } \quad0\leq t<\sqrt{\dfrac{1}{3\kappa}},\\
\dfrac{1}{3\sqrt{2 \kappa}t}+\sqrt{\dfrac{1}{6}},&\quad \text{ if }\quad \sqrt{\dfrac{1}{3\kappa}}\leq t
.\end{cases}\end{equation*}
Setting $g(t)=g(-t)$ for all $t\leq 0,$ it follows that $g \in C^1(\mathbb{R},(\sqrt{\frac{1}{6}},1])$, $g$ is a even function, increases in $(-\infty,0)$ and  decreases in $[0,+\infty)$.

Note that (\ref{1.0}) is the Euler-Lagrange equation associated to the natural energy
functional
 \begin{equation}\label{2.2}I_\kappa(u)=\frac{1}{2}\int_{\mathbb{R}^N}g^2(u)|\nabla u|^2dx+\frac{1}{2}\int_{\mathbb{R}^N}V(x)|u|^2dx-\frac{1}{q}\int_{\mathbb{R}^N}|u|^qdx,
\end{equation}

Our goal is proving the existence of a nontrivial critical point $u$ of (\ref{2.2}) satisfying $\sup\limits_{x\in \mathbb{R}^N}|u(x)|\leq \sqrt{\dfrac{1}{3\kappa}}$, which will be a  nontrivial solution of (\ref{1.0}) with $g(u)=\sqrt{1-\kappa u^2}$, and so, a nontrivial solution of (\ref{1.1}).

In what follows, we set
$$
G(t)=\int_{0}^{t}g(s)ds
$$
and we observe that inverse function $G^{-1}(t)$  exists and it is an odd function. Moreover, it is very important to observe that $G, G^{-1} \in C^{2}(\mathbb{R})$.

Next lemma shows important properties involving functions $g$ and $G^{-1}$ which will be used later on.

\begin{lemma}  \label{l2.1}  \begin{itemize}
\item [$(1)$] $\lim\limits_{t\rightarrow 0}\dfrac{G^{-1}(t)}{t}=1$;
\item [$(2)$] $\lim\limits_{t\rightarrow \infty}\dfrac{G^{-1}(t)}{t}=\sqrt{6}$;
\item [$(3)$] $t\leq G^{-1}(t)\leq \sqrt{6}t,$ for all $t\geq 0;$
\item [$(4)$] $-\dfrac{1}{2}\leq\dfrac{t}{g(t)}g'(t)\leq 0,$ for all $t\geq 0.$
\end{itemize}\end{lemma}
\begin{proof} By the definition of $g$, 
$$
\lim\limits_{t\rightarrow 0}\dfrac{G^{-1}(t)}{t}=\lim\limits_{t\rightarrow 0}\dfrac{1}{g(G^{-1}(t))}=1
$$
and
$$
\lim\limits_{t\rightarrow \infty}\dfrac{G^{-1}(t)}{t}=\lim\limits_{t\rightarrow \infty}\dfrac{1}{g(G^{-1}(t))}=\sqrt{6}.
$$ 
Thus, (1) and (2) are proved. Since $g(t)>0$ is decreasing in $[0,\infty)$, then $\sqrt{\dfrac{1}{6}}t\leq g(t)t\leq G(t)\leq t$ for all $t\geq0$, which implies  (3).  By a direct calculation,  we get (4).

\end{proof}

Now, fixing the change variable
\begin{equation} \label{Variable}
v=G(u)=\int_{0}^{u}g(s)ds,
\end{equation}
we observe that functional $I_\kappa$ can be written of the following way
\begin{equation}\label{1.7}
J_\kappa(v)=\frac{1}{2}\int_{\mathbb{R}^N} |\nabla
v|^2dx+\frac{1}{2}\int_{\mathbb{R}^N}
V(x)|G^{-1}(v)|^2dx-\frac{1}{q}\int_{\mathbb{R}^N}
|G^{-1}(v)|^qdx.\end{equation}
 From Lemma \ref{l2.1},  $J_\kappa$ is well defined in $H^1(\mathbb{R}^N)$, $J\in C^1(H^1(\mathbb{R}^N), \mathbb{R})$ and
\begin{equation}\label{1.9}
J_\kappa'(v)\psi=\int_{\mathbb{R}^N} \Big[\nabla v\nabla \psi + V(x)\frac{G^{-1}(v)}{g(G^{-1}(v))}\psi -\frac{|G^{-1}(v)|^{q-2}G^{-1}(v)}{g(G^{-1}(v))}\psi
\Big]dx,
\end{equation}
for all $v, \phi \in H^{1}(\mathbb{R}^{N})$.

\begin{lemma} \label{L0}
If $v \in C^{2}(\mathbb{R}^{N}) \cap H^{1}(\mathbb{R}^{N})$ is a critical point of $J_\kappa$, then \linebreak $u=G^{-1}(v) \in C^{2}(\mathbb{R}^{N}) \cap H^{1}(\mathbb{R}^{N})$ and it is a classical solution for (\ref{1.0}).
\end{lemma}
\begin{proof}

By using the fact that $G^{-1} \in C^{2}(\mathbb{R})$ together with Lemma \ref{l2.1}, a direct computation gives $u=G^{-1}(v)$ belongs to  $C^{2}(\mathbb{R}^{N}) \cap H^{1}(\mathbb{R}^{N})$.

If $v$ is a critical point for $J_\kappa$, we have that
\[
\int_{\mathbb{R}^N} \Big[\nabla v\nabla \psi + V(x)\frac{G^{-1}(v)}{g(G^{-1}(v))}\psi -\frac{|G^{-1}(v)|^{q-2}G^{-1}(v)}{g(G^{-1}(v))}\psi
\Big]dx=0,\quad \forall\psi \in H^{1}(\mathbb{R}^N).
\]
For each $\varphi \in C^\infty_0(\mathbb{R}^N) $, we consider $\psi=g(u)\varphi \in C_0^{2}(\mathbb{R}^N) \subset H^{1}(\mathbb{R}^{N})$  in (\ref{1.9}), to get
\[
\int_{\mathbb{R}^N} [g^2(u)\nabla u\nabla \varphi + g(u)g'(u)|\nabla u|^2 \varphi  +
V(x) u\varphi - |u|^{q-2}u\varphi ]dx =0
\]
or equivalently,
\[
\int_{\mathbb{R}^N} [-div(g^2(u)\nabla u) + g(u)g'(u)|\nabla u|^2 +
V(x) u-|u|^{q-2}u]\varphi dx =0 \,\,\, \forall \varphi \in C^{\infty}_{0}(\mathbb{R}^{N}),
\]
showing that $u$ is a classical solution of
$$
-div(g^2(u)\nabla u) + g(u)g'(u)|\nabla u|^2 +
V(x) u =|u|^{q-2}u \,\,\,\, \mbox{in} \,\,\,\, \mathbb{R}^{N}.
$$
\end{proof}

Therefore, in order to find a nontrivial solutions of (\ref{1.0}), it suffices studying the existence of nontrivial solutions  of  the following equation
\begin{eqnarray} \label{1.00}
-\Delta v+V(x)\frac{G^{-1}(v)}{g(G^{-1}(v))}-\frac{|G^{-1}(v)|^{q-2}G^{-1}(v)}{g(G^{-1}(v))}=0 \,\,\,\, \mbox{in} \,\,\,\, \mathbb{R}^N.\end{eqnarray}

Next, we establish the geometric hypotheses of the Mountain Pass Theorem for $J_\kappa$.
\begin{lemma}  \label{L2.1}
 For $2<q< 2^*$, there exist $\rho_0,$ $ a_0>0$, such that
$J_\kappa(v)\geq a_0$ for $\|v\|=\rho_0.$ Moreover, there exists $e\in H^1(\mathbb{R}^N)$ such that $J_\kappa(e)<0$.
\end{lemma}
\begin{proof}
By Lemma \ref{l2.1}$-$(3) and Sobolev embedding,
 \begin{equation*}
\begin{split}
J_\kappa(v)=&\frac{1}{2}\int_{\mathbb{R}^N} |\nabla
v|^2dx+\frac{1}{2}\int_{\mathbb{R}^N}
 V(x)|G^{-1}(v)|^2dx-\frac{1}{q} \int_{\mathbb{R}^N}
|G^{-1}(v)|^{q}dx\\
\geq&  \frac{1}{2}\int_{\mathbb{R}^N} |\nabla
v|^2dx+\frac{1}{2}\int_{\mathbb{R}^N} V(x)|v|^2dx-\frac{6^{q/2}}{q} \int_{\mathbb{R}^N}
|v|^{q}dx\\
\geq& \frac{1}{2}\|v\|^2-C\|v\|^{q}.
\end{split}\end{equation*}
Thereby, by choosing $\rho_0$ small, we get
\[
a_0= \frac{1}{2}\rho_0^2-C\rho_0^{q}>0,
\]
and so,
\[
J_\kappa(v)\geq a_0 \,\,\, \mbox{for} \,\,\, \|v\|=\rho_0.
\]

In order to prove the existence of $e\in H^1(\mathbb{R}^N)$ such that $J_\kappa(e)<0$, we fix $\varphi \in C_0^\infty(\mathbb{R}^N,[0,1])$ with $supp \varphi=\bar{B}_1$ and show that $J_\kappa(t\varphi)\rightarrow -\infty$ as $t\rightarrow \infty$, because the result follows taking $e=t\varphi$ with $t$ large enough. By Lemma \ref{l2.1}$-$(3),
\begin{equation*}\label{eq:barwq}
\begin{split}
J_\kappa(t \varphi)=&\frac{1}{2}t^2\int_{\mathbb{R}^N} |\nabla
\varphi|^2dx+\frac{1}{2}\int_{\mathbb{R}^N}
V(x)|G^{-1}(t\varphi)|^2dx-\frac{1}{q}\int_{\mathbb{R}^N}
|G^{-1}(t\varphi )|^{q}dx\\ \leq&3t^2\int_{\mathbb{R}^N}( |\nabla
\varphi|^2+V_\infty
\varphi^2)dx-\frac{1}{q}t^{q}\int_{\mathbb{R}^N} \varphi^{q} dx. \end{split}\end{equation*}
Since $q>2$, it follows that $J_\kappa(t\varphi)\rightarrow -\infty$ as $t\rightarrow \infty$. \end{proof}

\vspace{0.5 cm}

In consequence of Lemma \ref{L2.1} and Ambrosetti--Rabinowitz Mountain Pass Theorem \cite{Sch}, for the constant
\begin{eqnarray}\label{c}
c_\kappa=\inf\limits_{\gamma \in \Gamma_\kappa}\sup\limits_{t\in [0,1]}J_\kappa(\gamma(t))\geq a_0>0,
\end{eqnarray}
where
\[
\Gamma_\kappa=\{\gamma\in C([0,1],H^1(\mathbb{R}^N)): \gamma(0)=0, \gamma(1)\not=0, J_\kappa(\gamma(1))<0\},
\]
there exists a Palais-Smale sequence at level $c_\kappa$, that is,
\[
J_\kappa(v_n)\rightarrow c_\kappa \,\,\, \mbox{and} \,\,\, J_\kappa'(v_n)\rightarrow 0  \,\,\, \mbox{as} \,\,\, n\rightarrow \infty.
\]
\begin{lemma} \label{L3.3} For $2<q< 2^*$,
the Palais-Smale sequence $\{v_n\}$  is bounded.
\end{lemma}
\begin{proof} Since $\{v_n\}\subset H^1(\mathbb{R}^N)$ is a  Palais-Smale sequence, then
\begin{equation}\label{3.15}\begin{split}
J_\kappa(v_n)=&\frac{1}{2}\int_{\mathbb{R}^N} |\nabla
v_n|^2dx+\frac{1}{2}\int_{\mathbb{R}^N}
V(x)|G^{-1}(v_n)|^2dx-\dfrac{1}{q}\int_{\mathbb{R}^N}
|G^{-1}(v_n)|^{q}dx\\=&c_\kappa+o(1),\end{split}
\end{equation}
and for any $\psi \in H^{1}(\mathbb{R}^N)$,  $J_\kappa'(v_n)\psi=o(1)\|\psi\|$, that is,
\begin{equation}\label{3.2}
\int_{\mathbb{R}^N} \Bigg[\nabla v_n\nabla \psi + V(x)\frac{G^{-1}(v_n)}{g(G^{-1}(v_n))}\psi - \frac{|G^{-1}(v_n)|^{q-2}G^{-1}(v_n)}{g(G^{-1}(v_n))}\psi
\Bigg]dx =o(1)\|\psi\|.
\end{equation}
Fixing $\psi=G^{-1}(v_n)g(G^{-1}(v_n))$, it follows from Lemma \ref{l2.1}$-$(4),
\begin{eqnarray}\label{3.2h}|\nabla (G^{-1}(v_n)g(G^{-1}(v_n)))|\leq \Big[1+\dfrac{G^{-1}(v_n)}{g(G^{-1}(v_n))}g'(G^{-1}(v_n))\Big]|\nabla v_n|\leq |\nabla v_n|.\end{eqnarray}On the other hand, by   Lemma \ref{l2.1}$-$(3),
\begin{eqnarray}\label{3.2hhh}|G^{-1}(v_n)g(G^{-1}(v_n))|\leq  \sqrt{6}|v_n|.  \end{eqnarray}
Combining (\ref{3.2h}) and (\ref{3.2hhh}),  we have  $\psi \in H^1(\mathbb{R}^N)$ with $ \| \psi \|\leq 6\|v_n\|.$  Thus,  by  using $\psi=G^{-1}(v_n)g(G^{-1}(v_n))$ as a test function in (\ref{3.2}), we  derive that
\begin{eqnarray}\label{3.16}\begin{split}
o(1)\|v_n\| =&J_\kappa'(v_n)G^{-1}(v_n)g(G^{-1}(v_n))\\=&\int_{\mathbb{R}^N} \Big[\Big(1+\dfrac{G^{-1}(v_n)}{g(G^{-1}(v_n))}g'(G^{-1}(v_n))\Big)|\nabla v_n|^2 + V(x)|G^{-1}(v_n)|^2\\&-|G^{-1}(v_n)|^{q}\Big]dx\\ \leq &\int_{\mathbb{R}^N} \Big[ |\nabla v_n|^2 + V(x)|G^{-1}(v_n)|^2-|G^{-1}(v_n)|^{q}
\Big]dx.\end{split}
\end{eqnarray}
 Therefore, by (\ref{3.15}), (\ref{3.2}) and (\ref{3.16}),
\begin{equation*}\begin{split}
qc_\kappa+o(1)+o(1)\|v_n\|=&q J_\kappa(v_n)-J_\kappa'(v_n)G^{-1}(v_n)g(G^{-1}(v_n))\\ \geq &\frac{(q-2)}{2}\int_{\mathbb{R}^N}\Big[ |\nabla
v_n|^2+ V(x)|G^{-1}(v_n)|^2\Big]dx\\
\geq &\frac{(q-2)}{2}\|v_n\|^2,
\end{split}\end{equation*}
showing the boundedness of $\{v_n\}$.
\end{proof}

Since $\{v_n\}$ is a bounded sequence and $H^1(\mathbb{R}^N)$ is a separable Hilbert space, there exists $v_\kappa\in H^1(\mathbb{R}^N)$ and a subsequence of $\{v_n\}$, still denoted by itself, such  that
\[
v_n\rightharpoonup v_\kappa \,\,\, \mbox{in} \,\,\, H^1(\mathbb{R}^N), v_n \rightarrow v_\kappa \,\,\ \mbox{in} \,\,\,  L^q_{loc}(\mathbb{R}^N) \,\,\, \mbox{for} \,\,\, q\in [2,2^*) \,\,\ \mbox{and} \,\,\, v_n\rightarrow v_\kappa \,\,\, \mbox{ a.e. on} \,\,\, \mathbb{R}^N.
\]

\begin{theorem} \label{criticialpoint} The weak limit $v_\kappa$ of $\{v_n\}$ is a nontrivial critical point of $J_\kappa$ and $J_\kappa(v_\kappa) \leq c_\kappa$.
\end{theorem}

\begin{proof} Our first goal is proving that  $v_\kappa$ is a weak solution. To this end, it suffices showing that
\[
J_\kappa'(v_\kappa)\psi=0 \,\,\, \forall \psi \in H^{1}(\mathbb{R}^N)
\]
or equivalently,
\[
\int_{\mathbb{R}^N} \Big[\nabla v_\kappa\nabla \psi + V(x)\frac{G^{-1}(v_\kappa)}{g(G^{-1}(v_\kappa))}\psi -\frac{|G^{-1}(v_\kappa)|^{q-2}G^{-1}(v_\kappa)}{g(G^{-1}(v_\kappa))}\psi
\Big]dx=0 \,\,\, \forall \psi \in H^{1}(\mathbb{R}^N).
\]
Once that $C_0^\infty(\mathbb{R}^N)$ is dense in $H^{1}(\mathbb{R}^N)$, it is sufficient to show the last equality only for functions belonging to $C_0^\infty(\mathbb{R}^N)$.

In what follows, for each $R>0$ we consider $\psi_R \in C_0^\infty(\mathbb{R}^N)$ verifying
\[
0 \leq \psi_R(x) \leq 1 \,\,\, \forall x \in \mathbb{R}^N, \,\,\, \psi_R(x)=1 \,\, \forall x \in B_R(0) \,\,\, \mbox{and} \,\,\, \psi_R(x)=0 \,\, \forall x \in B^{c}_{2R}(0).
\]
By \cite{Willem-1996},
\[
|v_n|\leq |z(x)| \,\,\, \mbox{for every} \,\, n \,\, \mbox{with} \,\,\, z \in L^q(B_{2R}(0)).
\]
Consequently,
\begin{eqnarray*}\begin{split}
\frac{G^{-1}(v_n)}{g(G^{-1}(v_n))}v_n \rightarrow \frac{G^{-1}(v_\kappa)}{g(G^{-1}(v_\kappa))}v_\kappa \mbox{ a.e. on } B_{2R}(0), \,\, \mbox{as } n\rightarrow \infty,\end{split}
\end{eqnarray*}
and \begin{eqnarray*}\begin{split}
\frac{|G^{-1}(v_n)|^{q-2}G^{-1}(v_n)}{g(G^{-1}(v_n))}v_n \rightarrow \frac{|G^{-1}(v_\kappa)|^{q-2}G^{-1}(v_\kappa)}{g(G^{-1}(v_\kappa))}v_\kappa \mbox{ a.e. on } \, B_{2R}(0), \,\, \mbox{as } n\rightarrow \infty. \end{split}
\end{eqnarray*}
Moreover, by Lemma \ref{l2.1},
\begin{eqnarray*}\begin{split}
\Big|V(x)\frac{G^{-1}(v_n)}{g(G^{-1}(v_n))}v_n\psi \Big|\leq 6 V_\infty|v_n|^{2}|\psi|\leq 6V_\infty|z(x)|^2|\psi|\end{split}
\end{eqnarray*}
and
\begin{eqnarray*}\begin{split}
\Bigg|V(x)\frac{|G^{-1}(v_n)|^{q-2}G^{-1}(v_n)}{g(G^{-1}(v_n))}\psi \Bigg|\leq 6^{\frac{q+1}{2}}V_\infty|v_n|^{q}|\psi|\leq 6^{\frac{q+1}{2}}V_\infty|z(x)|^{q}|\psi|.\end{split}
\end{eqnarray*}
Hence, by Lebesgue Dominated Theorem
\begin{equation} \label{E1}
\int_{\mathbb{R}^N}V(x)\frac{G^{-1}(v_n)}{g(G^{-1}(v_n))}v_n\psi\,dx \to \int_{\mathbb{R}^N} V(x)\frac{G^{-1}(v_\kappa)}{g(G^{-1}(v_\kappa))}v_\kappa\psi \,dx
\end{equation}
and
\begin{equation} \label{E2}
\int_{\mathbb{R}^N} \frac{|G^{-1}(v_n)|^{q-2}G^{-1}(v_n)}{g(G^{-1}(v_n))}v_n\psi\,dx \to \int_{\mathbb{R}^N}\frac{|G^{-1}(v_\kappa)|^{q-2}G^{-1}(v_\kappa)}{g(G^{-1}(v_\kappa))}v_\kappa\psi \,dx.
\end{equation}
The same type of arguments also can be use to prove the below limits
\begin{equation} \label{E3}
\int_{\mathbb{R}^N}V(x)\frac{G^{-1}(v_n)}{g(G^{-1}(v_n))}v_\kappa \psi\,dx \to \int_{\mathbb{R}^N} V(x)\frac{G^{-1}(v_\kappa)}{g(G^{-1}(v_\kappa))}v_\kappa\psi \,dx
\end{equation}
and
\begin{equation} \label{E4}
\int_{\mathbb{R}^N} \frac{|G^{-1}(v_n)|^{q-2}G^{-1}(v_n)}{g(G^{-1}(v_n))}v_\kappa \psi\,dx \to \int_{\mathbb{R}^N}\frac{|G^{-1}(v_\kappa)|^{q-2}G^{-1}(v_\kappa)}{g(G^{-1}(v_\kappa))}v_\kappa \psi \,dx.
\end{equation}
Now, the above limits combined with $J_\kappa'(v_n)(v_n\psi)=o_n(1)$ and $J_\kappa'(v_n)(v_\kappa\psi)=o_n(1)$ give
\[
\int_{\mathbb{R}^{N}}|\nabla v_n - \nabla v_\kappa|^{2}\psi_R(x) \, dx \to 0,
\]
from where it follows that
\[
\int_{B_R(0)}|\nabla v_n - \nabla v_\kappa|^{2} \, dx \to 0.
\]
Once that $R$ is arbitrary and $v_n \to v_\kappa$ in $L^{2}_{loc}(\mathbb{R}^{N})$, we conclude that $v_n \to v_\kappa$ in $H^{1}_{loc}(\mathbb{R}^{N})$. Thereby,
\[
J_\kappa'(v_n)\psi \to J_\kappa'(v_\kappa)\psi \,\,\ \forall \psi \in C_0^\infty(\mathbb{R}^N).
\]
Since $J_\kappa'(v_n)\psi=o_n(1)$, the last limit yields $J_\kappa'(v_\kappa)\psi=0$ for all $\psi \in C_0^\infty(\mathbb{R}^N)$, showing that $v_\kappa$ is a critical point for $J_\kappa$.

Now, next step is showing that $v_\kappa\not\equiv 0$. To prove this, we argue by contradiction supposing that $v_\kappa=0$. We claim that in this
 case, $\{v_n\}$ is also a Palais-Smale  sequence for functional $J_{\kappa,\infty}:H^1(\mathbb{R}^N)\rightarrow \mathbb{R}$ defined by
 \begin{equation} \begin{split}
 J_{\kappa,\infty}(v)=&\frac{1}{2}\int_{\mathbb{R}^N} |\nabla
 v|^2dx+\frac{1}{2}V_\infty\int_{\mathbb{R}^N}
 |G^{-1}(v)|^2dx-\frac{1}{q}\int_{\mathbb{R}^N}
 |G^{-1}(v)|^qdx. \end{split}\end{equation}
 Indeed, since $V(x)\rightarrow V_\infty$ as $|x|\rightarrow \infty$, $|G^{-1}(s)|\leq \sqrt{6}|s|$  and $v_n\rightarrow 0$ in $L^2_{loc}(\mathbb{R}^N)$,  we have
 \begin{equation}
 J_\kappa(v_n)-J_{\kappa,\infty}(v_n)=\frac{1}{2}\int_{\mathbb{R}^N}\Big[V(x)-V_\infty\Big]
 |G^{-1}(v_n)|^2dx\rightarrow 0.\end{equation}
On the other hand, recalling $\dfrac{|G^{-1}(s)|}{g(G^{-1}(s))}\leq 6|s|$, we have
 \begin{equation}\begin{split}
 \sup\limits_{\|\psi\|\leq 1} |\langle J_\kappa'(v_n)-J_{\kappa,\infty}'(v_n),\psi\rangle |=\sup\limits_{\|\psi\|\leq 1}\Big|\int_{\mathbb{R}^N}\Big[V(x)-V_\infty\Big]\frac{G^{-1}(v_n)}{g(G^{-1}(v_n))}\psi dx\Big|\rightarrow 0.
 \end{split}\end{equation}
 Next, we claim that for all $R>0$, the following vanishing  cannot occurs:
 \begin{equation}\label{3.10}
 \lim\limits_{n\rightarrow \infty}\sup\limits_{y\in \mathbb{R}^N}\int_{B_R(y)}|v_n|^2dx=0.
 \end{equation}
Suppose by contradiction that (\ref{3.10}) occurs,   then by a Lions' compactness lemma \cite{Lions},  $v_n\rightarrow 0$ in $L^q( \mathbb{R}^N)$ for any $q\in (2, 2^*)$.
So,
 \begin{eqnarray*}
 \lim\limits_{n\rightarrow \infty}\int_{\mathbb{R}^N}|G^{-1}(v_n)|^{q}dx\leq 6^{\frac{q}{2}}\lim\limits_{n\rightarrow }\int_{\mathbb{R}^N}|v_n|^qdx\rightarrow 0
 \end{eqnarray*}
and
 \begin{eqnarray*}
 \lim\limits_{n\rightarrow \infty}\int_{\mathbb{R}^N}\dfrac{|G^{-1}(v_n)|^{q-2}G^{-1}(v_n)}{g(G^{-1}(v_n))}v_ndx\leq 6^{\frac{q+1}{2}}\lim\limits_{n\rightarrow }\int_{\mathbb{R}^N}|v_n|^qdx\rightarrow 0.
 \end{eqnarray*}
Once that,
 \begin{eqnarray*}
 \lim\limits_{s\rightarrow 0}\dfrac{1}{s^2}\Bigg[|G^{-1}(s)|^2-\dfrac{G^{-1}(s)}{g(G^{-1}(s))}s\Bigg]=
 \lim\limits_{s\rightarrow \infty}\dfrac{1}{|s|^q}\Bigg[|G^{-1}(s)|^2-\dfrac{G^{-1}(s)}{g(G^{-1}(s))}s\Bigg]=0,
 \end{eqnarray*}
 we derive
 \begin{eqnarray*}\label{+3.0}
 \lim\limits_{n\rightarrow \infty}\int_{\mathbb{R}^N}\Big[|G^{-1}(v_n)|^2-\dfrac{G^{-1}(v_n)}{g(G^{-1}(v_n))}v_n\Big]dx=0.
 \end{eqnarray*}
Therefore, we deduce that
 \begin{eqnarray*} \begin{split} 2c_\kappa+o(1)=&2J_\kappa(v_n)-J_\kappa'(v_n)v_n\\=&\int_{\mathbb{R}^N}\Big[|G^{-1}(v_n)|^2-\dfrac{G^{-1}(v_n)}{g(G^{-1}(v_n))}v_n\Big]dx\\
 &-\frac{2}{q}\int_{\mathbb{R}^N}|G^{-1}(v_n)|^{q}dx+\int_{\mathbb{R}^N}\dfrac{|G^{-1}(v_n)|^{q-2}G^{-1}(v_n)}{g(G^{-1}(v_n))}v_ndx \rightarrow 0,\\
 \end{split}
 \end{eqnarray*}
  which is a contradiction, because $c_\kappa\geq a_0>0$.

  Thus, $\{v_n\}$ does not vanish and there exist $\alpha, R>0$, and $\{y_n\}\subset \mathbb{R}^N$ verifying
 \begin{equation}\label{2.17+}
 \lim\limits_{n\rightarrow \infty}\int_{B_R(y_n)}|v_n|^2dx\geq \alpha>0.
 \end{equation}
 Define $\widetilde{v}_n(x)=v_n(x+y_n)$. Since $\{v_n\}$ is a Palais-Smale sequence for $J_{\kappa,\infty}$, $\{\widetilde{v}_n\}$ is also a Palais-Smale sequence for $J_{\kappa,\infty}$. Arguing as in the case of $\{v_n\}$, we get that $\widetilde{v}_n \to \widetilde{v}_\kappa$ in $H^1_{loc}(\mathbb{R}^N)$ and $J_{\kappa,\infty}'(\widetilde{v}_\kappa)=0.$ Moreover, by  (\ref{2.17+}),  we also have $\widetilde{v}_\kappa\not=0.$  Henceforward, without loss of generality, we assume that
\[
\widetilde{v}_n (x) \to \widetilde{v}_\kappa(x) \,\,\, \mbox{and} \,\,\, \nabla \widetilde{v}_n(x) \to \nabla \widetilde{v}_\kappa(x) \,\,\, \mbox{a.e. on } \,\, \mathbb{R}^{N}.
\]
The last limits together with Fatous' Lemma lead to
\begin{equation} \label{EX} \begin{split}
2c_\kappa=&\limsup\limits_{n\rightarrow \infty}[2 J_{\kappa,\infty}(\widetilde{v}_n)-J_{\kappa,\infty}'(\widetilde{v}_n)G^{-1}(v_n)g(G^{-1}(\widetilde{v}_n))]\\=&-\limsup\limits_{n\rightarrow \infty}\int_{\mathbb{R}^N}\frac{G^{-1}(\widetilde{v}_n)g'(G^{-1}(\widetilde{v}_n))}{g(G^{-1}(\widetilde{v}_n))}|\nabla \widetilde{v}_n|^2dx\\&-\frac{(2-q)}{q}\limsup\limits_{n\rightarrow \infty}\int_{\mathbb{R}^N}
|G^{-1}(\widetilde{v}_n)|^qdx\\
\geq&-\int_{\mathbb{R}^N}\frac{G^{-1}(\widetilde{v}_\kappa)g'(G^{-1}(\widetilde{v}_\kappa))}{g(G^{-1}(\widetilde{v}_\kappa))}|\nabla \widetilde{v}_\kappa|^2     dx-\frac{(2-q)}{q}\int_{\mathbb{R}^N}
|G^{-1}(\widetilde{v}_\kappa)|^qdx\\
=&2 J_{\kappa,\infty}(\widetilde{v}_\kappa)-J_{\kappa,\infty}'(\widetilde{v}_\kappa)G^{-1}(\widetilde{v}_\kappa)g(G^{-1}(\widetilde{v}_\kappa))\\=&2 J_{\kappa,\infty}(\widetilde{v}_\kappa),
\end{split}
\end{equation}
that is, $J_{\kappa,\infty}(\widetilde{v}_\kappa)\leq c_\kappa.$ Now, as in \cite{Jean-Tan}, we  define
 \begin{eqnarray*}\label{3.21}
 \widetilde{v}_{\kappa,t}(x)=\begin{cases} \widetilde{v}_\kappa(x/t),\quad &\text{ if } t>0,\\
 0, \quad &\text{ if } t=0.
 \end{cases}
 \end{eqnarray*}
 Then,
 \begin{eqnarray*}\int_{\mathbb{R}^N}|\nabla \widetilde{v}_{\kappa,t}|^2dx=t^{N-2}\int_{\mathbb{R}^N}|\nabla \widetilde{v}_\kappa|^2dx,\quad \int_{\mathbb{R}^N}|G^{-1}(\widetilde{v}_{\kappa,t})|^2dx=t^{N}\int_{\mathbb{R}^N}|G^{-1}(\widetilde{v}_\kappa)|^2dx,
 \end{eqnarray*}
  and
 \begin{eqnarray*} \int_{\mathbb{R}^N}|G^{-1}(\widetilde{v}_{\kappa,t})|^qdx=t^{N}\int_{\mathbb{R}^N}|G^{-1}(\widetilde{v}_\kappa)|^qdx.
 \end{eqnarray*}
Since $J_\infty'(\widetilde{v}_\kappa)=0$, elliptic regularity implies that $\widetilde{v}_\kappa \in C^{2}(\mathbb{R}^{N})$. Hence, 
\[
\dfrac{d}{dt}J_{\kappa,\infty}(\widetilde{v}_{\kappa,t})\Big|_{t=1}=0,
\]
leading to
\begin{eqnarray}\label{5+.1}\begin{split}\dfrac{(N-2)}{2N}\int_{\mathbb{R}^N}|\nabla \widetilde{v}_\kappa|^2dx=&-\dfrac{V_\infty}{2}\int_{\mathbb{R}^N}|G^{-1}(\widetilde{v}_\kappa)|^2dx+\dfrac{1}{q}\int_{\mathbb{R}^N}|G^{-1}(\widetilde{v}_\kappa)|^qdx.
 \end{split}\end{eqnarray}
Setting $\gamma(t)(x)=\widetilde{v}_{\kappa,t}(x)$, we see that
 \begin{eqnarray*}\begin{split}J_{\kappa,\infty}(\gamma(t))=&\dfrac{t^{N-2}}{2}\int_{\mathbb{R}^N}|\nabla \widetilde{v}_\kappa|^2dx-t^N\Big[-\dfrac{V_\infty}{2}\int_{\mathbb{R}^N}|G^{-1}(\widetilde{v}_\kappa)|^2dx
 \\&+\dfrac{1}{q}\int_{\mathbb{R}^N}|G^{-1}(\widetilde{v}_\kappa)|^{q}dx\Big].
 \end{split}\end{eqnarray*}
 Thus $\gamma \in C([0,\infty),H^1(\mathbb{R}^N))$ and
 \begin{eqnarray*}\begin{split}\dfrac{d}{dt}J_{\kappa,\infty}(\gamma(t))=&\dfrac{N-2}{2}t^{N-3}\int_{\mathbb{R}^N}|\nabla \widetilde{v}_\kappa|^2dx
 -Nt^{N-1}\Big[-\dfrac{V_\infty}{2}\int_{\mathbb{R}^N}|G^{-1}(\widetilde{v}_\kappa)|^2dx\\&+
 \dfrac{1}{q}\int_{\mathbb{R}^N}|G^{-1}(\widetilde{v}_\kappa)|^{q}dx\Big]\\
 =&\dfrac{(N-2)}{2}t^{N-3}(1-t^2)\int_{\mathbb{R}^N}|\nabla \widetilde{v}_\kappa|^2dx.
 \end{split}\end{eqnarray*}
 So, $\dfrac{d}{dt}J_{\kappa,\infty}(\gamma(t))>0$ for $t\in (0,1)$ and $\dfrac{d}{dt}J_{\kappa,\infty}(\gamma(t))<0$ for $t>1$ implying that
\[
\max\limits_{t \geq 0}J_{\kappa,\infty}({\gamma}(t))=J_{\kappa,\infty}(\widetilde{v}_\kappa).
\]
Furthermore, $J_{\kappa,\infty}(\gamma(L))<0$ for sufficiently large $L>1$, showing that \linebreak $\widehat{\gamma}(t)=\gamma(Lt)$ belongs to $\Gamma_\kappa$.
Thereby
$$
c_\kappa \leq \max\limits_{t\in [0,1]}J_\kappa(\gamma(t)):=J_\kappa(\gamma(\bar{t}))<J_{\kappa,\infty}(\gamma(\bar{t}))\leq \max\limits_{t\in [0,1]}J_{\kappa,\infty}(\gamma(t))=J_{\kappa,\infty}(\widetilde{v}_\kappa)\leq c_\kappa,
$$
which is a contradiction. This way, $v_\kappa$ is a nontrivial critical point for $J$. Moreover, repeating the same type of arguments explored in (\ref{EX}), we have that $J_\kappa(v_\kappa) \leq c_\kappa$. \end{proof}

\section{$L^\infty$  estimate  of the solution}

In this section, we will establish an $L^\infty$ estimate for  solution $v_\kappa$ obtained in Theorem \ref{criticialpoint}. Indeed, by standard elliptic regularity estimate \cite{Gi-Tr-01}, $v_\kappa \in L^\infty(\mathbb{R}^N)$. However, this boundedness is not enough to prove our results. In the following, we will prove an $L^\infty$ estimate dependent on $\kappa>0$. To this end, firstly  we need to give an uniform  boundedness of the Sobolev norm independent on $\kappa>0$ for $v_\kappa$ .
\begin{lemma}\label{L2.5} The solution $v_\kappa$ satisfies $\|v_\kappa \|^2 \leq \dfrac{2qc_\kappa}{q-2}$.
\end{lemma}
\begin{proof}  Using the hypothesis that $v_\kappa$ is a critical point of $J_\kappa$, 
\begin{equation*}\label{3.17}\begin{split}
qc_\kappa=&q J(v_\kappa)-J'(v_\kappa)G^{-1}(v_\kappa)g(G^{-1}(v_\kappa))\\ \geq &\frac{(q-2)}{2}\int_{\mathbb{R}^N} |\nabla
v_\kappa|^2dx+\frac{(q-2)}{2}\int_{\mathbb{R}^N} V(x)|G^{-1}(v_\kappa)|^2dx,\\
\end{split}\end{equation*}
from where it follows that,
\[
\|v_\kappa\|^2 \leq \frac{2qc_\kappa}{q-2}.
\]
\end{proof}

From now on, we consider the functional
\begin{eqnarray*}\label{3.5}\begin{split}P_\infty(v)=3\int_{\mathbb{R}^N}(|\nabla v|^2+V_\infty v^2)dx-\frac{1}{q}\int_{\mathbb{R}^N}|v|^qdx
\end{split}\end{eqnarray*}
and the set
$$
{\Gamma}^{\infty}=\{\gamma\in C([0,1],H^1(\mathbb{R}^N)): \gamma(0)=0, \gamma(1)\not=0, P_\infty(\gamma(1))<0\}.
$$
By Lemma \ref{l2.1}$-$(3), we have $J_\kappa(v)\leq P_\infty(v)$ and thus ${\Gamma}^{\infty}\subset \Gamma_\kappa.$ Therefore
$$
c_\kappa=\inf\limits_{\gamma \in \Gamma_\kappa}\sup\limits_{t\in [0,1]}J_\kappa(\gamma(t))\leq \inf\limits_{\gamma \in {\Gamma}^{\infty}}\sup\limits_{t\in [0,1]}J_\kappa(\gamma(t))\leq \inf\limits_{\gamma \in {\Gamma}^{\infty}}\sup\limits_{t\in [0,1]}P_\infty(\gamma(t)):=d_\infty,
$$
where $d_\infty$ is independent on $\kappa.$ Consequently, by Lemma \ref{L2.5},  the solution $v_\kappa$ must satisfy the estimate
\begin{equation} \label{Z6}
\|v_\kappa\|^2\leq \dfrac{2q{d_\infty}}{q-2}.
\end{equation}

\begin{proposition} \label{P4}
There exists a constant $C_0>0$ independent on $\kappa$, such that $\|v_\kappa\|_\infty\leq C_0 \kappa^{-\frac{1}{4}}$ for $\kappa\leq 6^{\frac{2}{q-2^*}}$.
\end{proposition}
\begin{proof}
In what follows, we denote $v_\kappa$ by $v$. For each $m\in \mathrm{N}$ and $\beta>1$, let $A_m=\{x\in \mathbb{R}^N:|v|^{\beta-1}\leq m\}$ and $B_m=\mathbb{R}^N\setminus A_m.$
Define
\begin{equation*}v_m=\begin{cases}
v|v|^{2(\beta-1)},&\quad \text{ in }  A_m,\\
m^2v,&\quad \text{ in } B_m.
\end{cases}\end{equation*}
Note that $v_m\in H^1(\mathbb{R}^N)$, $v_m\leq |v|^{2\beta-1}$ and
\begin{equation} \label{4.1}\nabla v_m=\begin{cases}
(2\beta-1)|v|^{2(\beta-1)}\nabla v,&\quad \text{ in }  A_m,\\
m^2\nabla v,&\quad \text{ in } B_m.
\end{cases}\end{equation}
Using $v_m$ as a test function in (\ref{1.00}), we deduce that
\begin{equation}\label{s3-}
\int_{\mathbb{R}^N}\Big[\nabla v \nabla v_m+V(x)\frac{G^{-1}(v)}{g(G^{-1}(v))}v_m\Big]dx=\int_{\mathbb{R}^N}\frac{|G^{-1}(v)|^{q-2}G^{-1}(v)}{g(G^{-1}(v))}v_mdx.
\end{equation}
By (\ref{s3-}),
\begin{equation}\label{s3}
\int_{\mathbb{R}^N}\nabla v \nabla v_mdx=(2\beta-1)\int_{A_m}|v|^{2(\beta-1)}|\nabla v|^2dx+m^2\int_{B_m}|\nabla v|^2dx.
\end{equation}
Let
\begin{equation*}w_m=\begin{cases}
v|v|^{\beta-1},&\quad \text{ in }  A_m,\\
mv,&\quad \text{ in } B_m.
\end{cases}\end{equation*}
Then $w_m^2=vv_m\leq |v|^{2\beta}$ and
\begin{equation*}\nabla w_m=\begin{cases}
\beta|v|^{\beta-1}\nabla v,&\quad \text{ in }  A_m,\\
m\nabla v,&\quad \text{ in } B_m.
\end{cases}\end{equation*}
Hence,
\begin{equation}\label{s4}
\int_{\mathbb{R}^N}|\nabla w_m|^2dx=\beta^2\int_{A_m}|v|^{2(\beta-1)}|\nabla v|^2dx+m^2\int_{B_m}|\nabla v|^2dx.
\end{equation}
Then, from (\ref{s3}) and (\ref{s4}),
\begin{equation}\label{s5}
\int_{\mathbb{R}^N}(|\nabla w_m|^2-\nabla v \nabla v_m)dx=(\beta-1)^2\int_{A_m}|v|^{2(\beta-1)}|\nabla v|^2dx.
\end{equation}
Combing (\ref{s3-}), (\ref{s3}) and (\ref{s5}), since $\beta>1$, we have
\begin{eqnarray*}\begin{split}
\int_{\mathbb{R}^N}|\nabla w_m|^2dx\leq&\Big[\frac{(\beta-1)^2}{2\beta-1}+1\Big]\int_{\mathbb{R}^N}\nabla v\nabla v_mdx\\
\leq& \beta^2\int_{\mathbb{R}^N}\Big[\nabla v\nabla v_m +V(x) \frac{G^{-1}(v)}{g(G^{-1}(v))}v_m     \Big]dx\\
=&\beta^2\int_{\mathbb{R}^N}\frac{|G^{-1}(v)|^{q-2}G^{-1}(v)}{g(G^{-1}(v))}v_mdx.
\end{split}\end{eqnarray*}
Choosing $\kappa^{\frac{2^*-q}{4}}\leq\sqrt{\frac{1}{6}}$, we have $g(t)\geq \kappa^{\frac{2^*-q}{4}}$. Setting $\theta=\frac{2^*-q}{4}$, by Sobolev inequality and Lemma \ref{l2.1}$-$(3),
\begin{eqnarray*}\begin{split}
\Big(\int_{A_m}| w_m|^{2^*}dx\Big)^{(N-2)/N}\leq S\int_{\mathbb{R}^N}|\nabla w_m|^2dx\leq& 6^{\frac{q-1}{2}}S\beta^2\kappa^{-\theta}\int_{\mathbb{R}^N}|v|^{q-2}w_m^2dx.
\end{split}\end{eqnarray*}
By H\"{o}lder inequality, we have
\begin{eqnarray*}\begin{split}
\Big(\int_{A_m}| w_m|^{2^*}dx\Big)^{(N-2)/N}\leq  6^{\frac{q-1}{2}}S\beta^2\kappa^{-\theta}\|v\|_{2^*}^{q-2}\Big(\int_{\mathbb{R}^N}|w_m|^{2q_1}dx\Big)^{1/q_1}
\end{split}\end{eqnarray*}
where $1/q_1+(q-2)/2^*=1$. Since $|w_m|\leq |v|^\beta$ in $\mathbb{R}^N$ and $|w_m|=|v|^\beta$ in $A_m$, we have
\begin{eqnarray*}\begin{split}
\Big(\int_{A_m}| v|^{\beta2^*}dx\Big)^{(N-2)/N}\leq  6^{\frac{q-1}{2}}S\beta^2\kappa^{-\theta}\|v\|_{2^*}^{q-2}\Big(\int_{\mathbb{R}^N}|v|^{2\beta q_1}dx\Big)^{1/q_1}
\end{split}\end{eqnarray*}
By Monotone Convergence Theorem, letting $m\rightarrow \infty,$ we have
\begin{eqnarray}\label{s7}\begin{split}
\|v\|_{\beta 2^*}\leq  \beta^{1/\beta}(6^{\frac{q-1}{2}}S\kappa^{-\theta}\|v\|_{2^*}^{q-2})^{1/(2\beta)}\|v\|_{2\beta q_1}.
\end{split}\end{eqnarray}
Setting $\sigma=2^*/(2q_1)$ and $\beta=\sigma$ in (\ref{s7}), we obtain $2q_1\beta=2^*$ and
\begin{eqnarray}\label{s8++}\begin{split}
\|v\|_{\sigma 2^*}\leq  \sigma^{1/\sigma}(6^{\frac{q-1}{2}}S\kappa^{-\theta}\|v\|_{2^*}^{q-2})^{1/(2\sigma)}\|v\|_{2^*}.
\end{split}\end{eqnarray}
Taking $\beta=\sigma^2$ in (\ref{s7}), we have
\begin{eqnarray}\label{s9}\begin{split}
\|v\|_{\sigma^2 2^*}\leq  \sigma^{2/\sigma^2}(6^{\frac{q-1}{2}}S\kappa^{-\theta}\|v\|_{2^*}^{q-2})^{1/(2\sigma^2)}\|v\|_{\sigma2^*}.
\end{split}\end{eqnarray}
From (\ref{s8++}) and (\ref{s9}), 
\begin{eqnarray*}\label{s10}\begin{split}
\|v\|_{\sigma^2 2^*}\leq  \sigma^{1/\sigma+2/\sigma^2}(6^{\frac{q-1}{2}}S\kappa^{-\theta}\|v\|_{2^*}^{q-2})^{1/2(1/\sigma+1/\sigma^2)}\|v\|_{2^*}.
\end{split}\end{eqnarray*}
Taking $\beta=\sigma^i \ (i=1,2,\cdots)$ and iterating (\ref{s7}), we get
\begin{eqnarray*}\label{s10}\begin{split}
\|v\|_{\sigma^j 2^*}\leq  \sigma^{\sum\limits_{i=1}^j\frac{i}{\sigma^i}}(6^{\frac{q-1}{2}}S\kappa^{-\theta}\|v\|_{2^*}^{q-2})^{\frac{1}{2}\sum\limits_{i=1}^j\frac{1}{\sigma^i}}\|v\|_{2^*}.
\end{split}\end{eqnarray*}
Therefore, by Sobolev inequality, (\ref{Z6})  and taking the limit of $j\rightarrow +\infty$,   we get
$$\|v\|_\infty\leq \sigma^{\frac{1}{(\sigma-1)^2}}(6^{\frac{q-1}{2}}\kappa^{-\theta}S^{q/2}C^{(q-2)/2})^{\frac{1}{2(\sigma-1)}}S^{1/2}C^{1/2}=C_0\kappa^{-\frac{1}{4}},  \mbox{ for } \kappa\leq 6^{\frac{2}{q-2^*}},$$
where $C_0>0$ is independent of $\kappa>0$. This ends the proof.
\end{proof}

\subsection{Proof of Theorem \ref{th1.1}. }
Combining the arguments in Section 2 and  Proposition \ref{P4}, the solution $v_\kappa$ of (\ref{1.0}) established in Theorem \ref{criticialpoint} satisfies $\|v_\kappa\|_\infty\leq C_0\kappa^{-\frac{1}{4}}$ for $\kappa\leq 6^{\frac{2}{q-2^*}}$. Choosing $\kappa_0=\min\Big\{6^{\frac{2}{q-2^*}},\dfrac{1}{C_0\sqrt{18}}\Big\}$, it follows that
\[
\|G^{-1}(v_\kappa)\|_\infty \leq \sqrt{6}\|v_\kappa\|_{\infty}< \sqrt{\dfrac{1}{3\kappa}} \,\,\ \forall \kappa \in [0, \kappa_0).
\]
From this, $u=G^{-1}(v_\kappa)$ is a classical solution of (\ref{1.1}).

\section{Proof of Theorem \ref{th1.2}}
In this section, we fix $0<\kappa<\frac{1}{3}$ and  for equation (\ref{1.0}), we let \begin{equation*}g(t)=\begin{cases}
\sqrt{1-\kappa t^2},&\quad \text{ if } \quad0\leq t<1<\sqrt{\dfrac{1}{3\kappa}},\\
\dfrac{\kappa }{t\sqrt{1-\kappa}}+\dfrac{1-2\kappa }{\sqrt{1-\kappa}},&\quad \text{ if }\quad  t\geq1
.\end{cases}\end{equation*}
Setting $g(t)=g(-t)$ for all $t\leq 0, $ clearly $g \in C^1(\mathbb{R},(\frac{1-2\kappa}{\sqrt{1-\kappa}},1])$, $g$ increases in $(-\infty,0)$ and decreases  in $[0,+\infty)$.

We further modify the nonlinearity of equation (\ref{1.1++})  as follows:
\begin{equation*}f(t)=\begin{cases}
0,&\quad \text{ if } \quad t\leq0,\\
\Big[1-\dfrac{1}{(1+t^2)^3}\Big]t,&\quad \text{ if }\quad 0\leq t\leq 1,\\
\dfrac{7}{8}t^{q-1},&\quad \text{ if }\quad t\geq 1,\end{cases}\end{equation*}
where $ 2 < q < \min\{\frac{14}{5},2^{*}\}$ and fix $l(t)=f(t)$ in (\ref{1.0}).
 We note that $f$ is continuous and satisfies the following conditions:
 \begin{itemize}
\item[$(f_1)$] $f(0)=0;$
\item[$(f_2)$] $\lim\limits_{t\rightarrow 0}\dfrac{f(t)}{t}=0;$
\item[$(f_3)$]  $\lim\limits_{t\rightarrow +\infty}\dfrac{f(t)}{t^{q-1}}=\dfrac{7}{8};$
\item[$(f_4)$]  $\lim\limits_{t\rightarrow +\infty}\dfrac{f(t)}{t}=+\infty;$
\item[$(f_5)$]  $2F(t)-f(t)t\leq 0$ with $ t\in \mathbb{R}$,  where  $F(t)=\int_{0}^tf(s)ds$;
\item[$(f_6)$]  $f(t)\leq 7t^{q-1}$ with $ t \geq 0$.
\end{itemize}
We make   change of variables $$v=G(u)=\int_0^ug(t)dt.$$ Then, at this moment, the inverse function $G^{-1}(t)$ satisfies the following properties:
\begin{lemma}  \label{l4.1}  \begin{itemize}
\item [$(1)$] $\lim\limits_{t\rightarrow 0}\dfrac{G^{-1}(t)}{t}=1$;
\item [$(2)$] $\lim\limits_{t\rightarrow \infty}\dfrac{G^{-1}(t)}{t}=\dfrac{\sqrt{1-\kappa}}{1-2\kappa}$;
\item [$(3)$] $t\leq G^{-1}(t)\leq 3t,$ for all $t\geq 0;$
\item [$(4)$] $-\dfrac{3}{2}\leq\dfrac{t}{g(t)}g'(t)\leq 0,$ for all $t\geq 0.$
\end{itemize}\end{lemma}
\begin{proof} By the definition of $g$ and since $\kappa<{\dfrac{1}{3}}$, 
$$\lim\limits_{t\rightarrow 0}\dfrac{G^{-1}(t)}{t}=\lim\limits_{t\rightarrow 0}\dfrac{1}{g(G^{-1}(t))}=1$$
and
$$\lim\limits_{t\rightarrow \infty}\dfrac{G^{-1}(t)}{t}=\lim\limits_{t\rightarrow \infty}\dfrac{1}{g(G^{-1}(t))}=\dfrac{\sqrt{1-\kappa}}{1-2\kappa}\leq3.$$ Thus, (1) and (2) are proved. Since $g(t)>0$ is decreasing in $[0,\infty)$, then $\dfrac{1}{3}t\leq \frac{1-2\kappa}{\sqrt{1-\kappa}}t\leq g(t)t\leq G(t)\leq t$ for all $t\geq0$, which implies  (3).
 By direct calculation,  we get (4).

\end{proof}

Next, we consider the equation 
\begin{eqnarray} \label{1.00+}
-\Delta v+V(x)\frac{G^{-1}(v)}{g(G^{-1}(v))}-\frac{f(G^{-1}(v))}{g(G^{-1}(v))}=0 \,\,\,\, \mbox{in} \,\,\,\, \mathbb{R}^N.
\end{eqnarray}
We establish the geometric hypotheses of the Mountain Pass Theorem for the following energy functional corresponding to (\ref{1.00+}):
\begin{equation}\label{4.7}
\widetilde{J}_\kappa(v)=\frac{1}{2}\int_{\mathbb{R}^N} |\nabla
v|^2dx+\frac{1}{2}\int_{\mathbb{R}^N}
V(x)|G^{-1}(v)|^2dx-\int_{\mathbb{R}^N}
F(G^{-1}(v))dx.\end{equation}
 From  Lemma \ref{l4.1} and $(f_1)$$-$$(f_3)$,  $\widetilde{J}_\kappa$ is well defined in $H^1(\mathbb{R}^N)$ and $\widetilde{J}_\kappa\in C^1(H^{1}(\mathbb{R}^{N}),\mathbb{R}).$
\begin{lemma}  \label{l4.2}
  \begin{itemize}\item [$(1)$] There exist $\rho_1,$ $ a_1>0$, such that
$\widetilde{J}_\kappa(v)\geq a_1$ for
$\|v\|=\rho_1.$
\item [$(2)$] There exists $\phi\in H^1(\mathbb{R}^N)$ such that $\widetilde{J}_\kappa(\phi)<0$.
\end{itemize}
\end{lemma}
\begin{proof}
By $(f_2)$ and $(f_3)$,
$$|F(t)|\leq  \frac{1}{4}V_0 |t|^2+C|t|^q.$$
Thus, by Lemma \ref{l4.1}$-$(3) and Sobolev embedding inequality,  we have
 \begin{equation*}
\begin{split}
\widetilde{J}_\kappa(v)=&\frac{1}{2}\int_{\mathbb{R}^N} |\nabla
v|^2dx+\frac{1}{2}\int_{\mathbb{R}^N}
 V(x)|G^{-1}(v)|^2dx- \int_{\mathbb{R}^N}
F(G^{-1}(v))dx\\
\geq&  \frac{1}{2}\int_{\mathbb{R}^N} |\nabla
v|^2dx+\frac{1}{2}\int_{\mathbb{R}^N}
 V(x)|G^{-1}(v)|^2dx- \frac{1}{4}V_0\int_{\mathbb{R}^N}
|G^{-1}(v)|^2dx\\&-C\int_{\mathbb{R}^N}
|G^{-1}(v)|^qdx\\
\geq&\frac{1}{2}\int_{\mathbb{R}^N} |\nabla
v|^2dx+\frac{1}{4}V_0\int_{\mathbb{R}^N}v^2dx-C \int_{\mathbb{R}^N}
|v|^{q}dx\\
\geq& \frac{1}{4}\|v\|^2-C\|v\|^{q}.
\end{split}\end{equation*}
Therefore, by choosing $\rho_1$ small, we get (1) for $\|v\|=\rho_1.$

To prove (2),  we choose some $\varphi \in C_0^\infty(\mathbb{R}^N,[0,1]) \setminus \{0\}$
with $supp \varphi=\bar{B}_1(0)$. We will show that
$\widetilde{J}_\kappa(t\varphi)\rightarrow -\infty$ as $t\rightarrow \infty$, which will
prove the result if we take $\phi=t\varphi$ with $t$ large enough. By $(f_4)$,
$$
\lim\limits_{t\rightarrow +\infty}\dfrac{F(t)}{t^2}=+\infty.
$$
Thus, given $A=\frac{9}{2}V_\infty \int_{\mathbb{R}^N}\varphi^2dx + \frac{1}{2}\int_{\mathbb{R}^N} |\nabla
\varphi|^2dx +1 $, there exists $D>0$ such that
$$
F(t) \geq At^{2} - D \,\,\ \forall t \geq 0.
$$
Hence,
\[
\widetilde{J}_\kappa(t \varphi)\leq - t^2 + D|B_1(0)| \rightarrow  -\infty \,\,\, \mbox{as} \,\,\, t \to +\infty,
\]
which implies (2).
\end{proof}

In consequence of Lemma \ref{l4.1} and of a special version of the Mountain Pass Theorem found in \cite{Ek}, for the constant
\begin{eqnarray}\label{cc}
\tilde{c}_\kappa=\inf\limits_{\gamma \in \tilde{\Gamma}_\kappa}\sup\limits_{t\in [0,1]}\widetilde{J}_\kappa(\gamma(t))\geq a_0>0,
\end{eqnarray} where 
$$
\tilde{\Gamma}_\kappa=\{\gamma\in C([0,1],H^1(\mathbb{R}^N)): \gamma(0)=0, \gamma(1)\not=0, \widetilde{J}_\kappa(\gamma(1))<0\},
$$ 
there exists a Cerami sequence at level $\tilde{c}_\kappa$, that is,
\begin{equation} \label{Z1}
\widetilde{J}_\kappa(v_n)\rightarrow c \,\,\, \mbox{and} \,\,\, (1+\|v_n\|)\| \widetilde{J}_\kappa'(v_n)\| \rightarrow 0.
\end{equation}

\begin{lemma}
The Cerami sequence $\{v_n\}$ given in (\ref{Z1}) is bounded.
\end{lemma}

\begin{proof} For any $v \in H^{1}(\mathbb{R}^{N})$, 
\begin{equation}\label{4.8.1+++}\begin{split}
\widetilde{J}_\kappa(v)=&\frac{1}{2}\int_{\mathbb{R}^N} |\nabla
v|^2dx+\frac{1}{2}\int_{\mathbb{R}^N}
V(x)|G^{-1}(v)|^2dx\\&-\int_{\{x\in\mathbb{R}^N,|G^{-1}(v(x))|\leq1\}}
\Big[\frac{G^{-1}(v)^2}{2}+\frac{1}{4(1+G^{-1}(v)^2)^2}-\frac{1}{4}\Big]dx\\
&-\frac{7}{8q}\int_{\{x\in\mathbb{R}^N,|G^{-1}(v(x))|\geq1\}}
|G^{-1}(v)|^qdx-\dfrac{(5q-14)}{16q}\int_{\{x\in\mathbb{R}^N,|G^{-1}(v(x))|\geq1\}}
dx
\end{split}\end{equation}
and
\begin{equation}\label{4.9.1++}\begin{split}
\mbox{} &\widetilde{J}_\kappa'(v)G^{-1}(v)g(G^{-1}(v))=\\&\int_{\mathbb{R}^N}\Big[1+\dfrac{G^{-1}(v)}{g(G^{-1}(v))}g'(G^{-1}(v))\Big] |\nabla
v|^2dx+\int_{\mathbb{R}^N}
V(x)|G^{-1}(v)|^2dx\\&-\int_{\{x\in\mathbb{R}^N,|G^{-1}(v(x))|\leq1\}}
\Big[G^{-1}(v)^2-\frac{G^{-1}(v_n)^2}{(1+G^{-1}(v)^2)^3}\Big]dx\\
&-\frac{7}{8}\int_{\{x\in\mathbb{R}^N,|G^{-1}(v(x))|\geq1\}}
|G^{-1}(v)|^qdx
.\end{split}\end{equation}

Now, by previous arguments, we know that there is $C>0$ such that
\[
\|G^{-1}(v_n)g(G^{-1}(v_n))\| \leq C \|v_n\| \,\,\, \forall n \in \mathbb{N}.
\]
Thus, the last inequality combined with  (\ref{Z1}), (\ref{4.8.1+++}) and (\ref{4.9.1++}) implies that
\begin{equation*}\label{4.10++}\begin{split}
&q\tilde{c_\kappa}+o_n(1)\\=&q\widetilde{J}_\kappa(v_n)-\widetilde{J}_\kappa'(v_n)G^{-1}(v)g(G^{-1}(v_n))\\\geq &\int_{\mathbb{R}^N}\Big[\frac{(q-2)}{2}-\dfrac{G^{-1}(v_n)}{g(G^{-1}(v_n))}g'(G^{-1}(v_n))\Big] |\nabla
v_n|^2dx+\\&+\frac{(q-2)}{2}\int_{\{x\in\mathbb{R}^N,|G^{-1}(v_n(x))|\leq1\}}
(V(x)-1)G^{-1}(v_n)^2dx\\&+\int_{\{x\in\mathbb{R}^N,|G^{-1}(v_n(x))|\leq1\}}
\Big[\frac{2q|G^{-1}(v_n)|^2+q|G^{-1}(v_n)|^4}{4(1+G^{-1}(v_n)^2)^2}-\frac{G^{-1}(v_n)^2}{(1+G^{-1}(v_n)^2)^3}\Big]dx\\
\geq& \frac{(q-2)}{2} \int_{\mathbb{R}^N} |\nabla
v_n|^2dx.\end{split}\end{equation*}
Once that $V(x) \geq  1$ for all $x \in \mathbb{R}^{N}$, $q>2$ and 
$$
\frac{2qt^2+qt^4}{4(1+t^2)^2}-\frac{t^2}{(1+t^2)^3} \geq 0\,\,\, \forall t \in [0,1],
$$
it follows that 
\begin{equation} \label{Z0}
\limsup_{n \to \infty}\|\nabla v_n\|_2^{2}\leq\dfrac{2q\tilde{c}_\kappa}{q-2}.
\end{equation}
Recalling that there is $S>0$ such that
\[
\int_{\mathbb{R}^N}|v|^{2^{*}} \leq S \left(\int_{\mathbb{R}^N}|\nabla v|^{2}dx \right)^{\frac{2^{*}}{2}} \,\,\, \forall v \in H^{1}(\mathbb{R}^N),
\]we derive that
\begin{equation} \label{Z2}
\limsup_{n \to +\infty}\int_{\mathbb{R}^N}|v_n|^{2^{*}} \leq \left( \dfrac{2q\tilde{c}_\kappa}{q-2}\right)^{2^{*}} \,\, \forall n \in \mathbb{N}.
\end{equation}
From definition of $f$, given $\epsilon=\frac{V_0}{2}$, there is $C>0$ such that
$$
f(t) \leq \frac{V_0}{18}t + Ct^{2^{*}-1} \,\,\, \forall t \geq 0.
$$
This together with $J'(v_n)v_n=o_n(1)$ gives
\[
\int_{\mathbb{R}^N} \Big[|\nabla v_n|^{2} + V(x)\frac{G^{-1}(v_n)}{g(G^{-1}(v_n))}v_n\Big]dx\leq \int_{\mathbb{R}^N}\Big[\frac{\frac{V_0}{18}(G^{-1}(v_n))+C(G^{-1}(v_n))^{2^{*}-1}}{g(G^{-1}(v_n))}\Big]v_n dx.
\]
Using Lemma \ref{l4.1}-3 and the fact that  $ \frac{1}{3} \leq g(t) \leq 1$ for all $t \in \mathbb{R}$, we get
\[
\frac{V_0}{2}\int_{\mathbb{R}^N}|v_n|^{2}dx \leq C \int_{\mathbb{R}^N}|v_n|^{2^{*}}.
\]
Then, by (\ref{Z2}),
\begin{equation} \label{Z3}
\limsup_{n \to \infty} \int_{\mathbb{R}^N}|v_n|^{2}dx \leq \frac{2C}{V_0} \left( \dfrac{2q\tilde{c}}{q-2}\right)^{2^{*}} \,\, \forall n \in \mathbb{N}.
\end{equation}
From (\ref{Z0}) and (\ref{Z2}), it follows that $\{v_n\}$ is bounded in $H^{1}(\mathbb{R}^{N})$.

\end{proof}

Analogous to the arguments in the end of Section 2,  we can assume there is $v_\kappa \in H^{1}(\mathbb{R}^{N})$ and a subsequence of $\{v_n\}$, still denoted by itself, such that
\[
v_n \rightharpoonup v_\kappa \,\, \mbox{in} \,\, H^{1}(\mathbb{R}^{N}), v_n \to v_\kappa \,\, \mbox{in} \,\, H_{loc}^{1}(\mathbb{R}^{N}) \,\,\, \mbox{and} \,\,\, v_n \to v_\kappa \,\, \mbox{in} \,\, L_{loc}^{p}(\mathbb{R}^N) \,\,\, \forall p \in [1,2^{*}).
\]
Moreover, $v_\kappa$ is a nontrivial critical point for $\widetilde{J}_\kappa$, $\widetilde{J}_\kappa(v_\kappa) \leq \tilde{c}_\kappa$ and by (\ref{Z0}),
\begin{equation} \label{Z5}
\|\nabla v_\kappa \|_2^{2}\leq\dfrac{2q\tilde{c}_\kappa}{q-2}.
\end{equation}

\begin{lemma}\label{L2.5} The nontrivial solution $v_\kappa$ of (\ref{1.00+}) verifies $\|\nabla v_\kappa\|_2\leq \dfrac{2qd_\infty}{q-2}$, where $d_\infty$ is independent on $\kappa>0$.
\end{lemma}
\begin{proof}
Setting the functional $Q_\infty : H^{1}(\mathbb{R}^{N}) \to \mathbb{R}$ given by
\begin{eqnarray*}\label{3.5}\begin{split}Q_\infty(v)=\frac{1}{2}\int_{\mathbb{R}^N} |\nabla
v|^2dx+\frac{9V_\infty}{2}\int_{\mathbb{R}^N}
|v|^2dx-\int_{\mathbb{R}^N}F(v)dx,
\end{split}\end{eqnarray*}
we observe that
\[
\widetilde{J}_\kappa(v) \leq Q_\infty(v) \,\,\, \forall v \in H^{1}(\mathbb{R}^{N}).
\]
As in Lemma \ref{l4.2}, it follows that $Q_\infty$ verifies the mountain pass geometry. Therefore, if $d_\infty$ denotes the mountain pass level associated with $Q_\infty$, the last inequality yields $\tilde{c}_\kappa \leq d_\infty$. Hence, by (\ref{Z5})
\[
\|\nabla v_\kappa \|_2^{2}\leq\dfrac{2qd_\infty}{q-2},
\]
finishing the proof.  \end{proof}

\begin{proposition} \label{PP4}
There exists a constant $C_1>0$ independent on $\kappa$ such that  $\|v_\kappa\|_\infty\leq C_1 \kappa^{\frac{1}{2(2^*-q)}}$.
\end{proposition}
\begin{proof} Since the proof  is similar to  Proposition \ref{P4}, we only indicate the necessary changes. Denoting $v_\kappa$ by $v$, we now have
\begin{eqnarray*}\begin{split}
\int_{\mathbb{R}^N}|\nabla w_m|^2dx\leq&\Big[\frac{(\beta-1)^2}{2\beta-1}+1\Big]\int_{\mathbb{R}^N}\nabla v\nabla v_mdx\\
\leq& \beta^2\int_{\mathbb{R}^N}\Big[\nabla v\nabla v_m +V(x) \frac{G^{-1}(v)}{g(G^{-1}(v))}v_m     \Big]dx\\
=&\beta^2\int_{\mathbb{R}^N}\frac{f(G^{-1}(v))}{g(G^{-1}(v))}v_mdx\\
\leq& {7}\beta^2\int_{\mathbb{R}^N}\frac{|G^{-1}(v)|^{q-1}}{g(G^{-1}(v))}v_mdx.
\end{split}\end{eqnarray*}
Since $\dfrac{1}{g(t)}\leq \dfrac{\sqrt{1-\kappa}}{1-2\kappa}\leq 3\sqrt{2\kappa}$ for $0<\kappa\leq \frac{1}{3}$.  Then, by Sobolev inequality and Lemma \ref{l2.1}$-$(3), we get
\begin{eqnarray*}\begin{split}
\Big(\int_{A_m}| w_m|^{2^*}dx\Big)^{(N-2)/N}\leq S\int_{\mathbb{R}^N}|\nabla w_m|^2dx\leq& 3^{q}\sqrt{2}S\beta^2\kappa^{\frac{1}{2}}\int_{\mathbb{R}^N}|v|^{q-2}w_m^2dx.
\end{split}\end{eqnarray*}
Therefore,
$$\|v\|_\infty\leq \sigma^{\frac{1}{(\sigma-1)^2}}(3^{q}\sqrt{2}S\beta^2\kappa^{\frac{1}{2}}S^{q/2}C^{(q-2)/2})^{\frac{1}{2(\sigma-1)}}S^{1/2}C^{1/2}=C_1\kappa^{\frac{1}{2(2^*-q)}}, $$
where $C_1>0$ is independent of $\kappa>0$, finishing the proof.
\end{proof}

\subsection{Proof of Theorem \ref{th1.2}. }
Combining the above arguments and  Proposition \ref{PP4}, the solution $v_\kappa$ of (\ref{1.0}) satisfies $\|v_\kappa\|_\infty\leq C_1\kappa^{\frac{1}{2(2^*-q)}}$. Fixing $\kappa_1=\min\Big\{\dfrac{1}{3},\left(\frac{1}{18C_1^{2}}\right)^{\frac{2^{*}-q}{2^{*}-q+1}}\Big\}$, we have that
\[
\|G^{-1}(v_\kappa)\|_{\infty}\leq \sqrt{6}\|v_\kappa\|_\infty < \sqrt{\dfrac{1}{3\kappa}} \,\,\ \forall \kappa \in [0, \kappa_1).
\]
This implies that $u=G^{-1}(v_\kappa)$ is a positive solution of (\ref{1.1++}).

\end{document}